\documentclass[a4paper,11pt,reqno]{amsart}
\usepackage{amssymb}
\usepackage{epsfig}
\usepackage{graphics,color}
\usepackage{tikz}
\usetikzlibrary{arrows}

\usepackage{hyperref}

\newcommand\bp{B}

\newcommand{\ds}{\displaystyle}

\newcommand{\om} {\Omega}
\newcommand{\ofm} {\Omega^{FM}}
\newcommand{\oal} {\Omega^C}
\newcommand{\con} {\mathcal Q}

\newcommand{\vac} {\ensuremath{\circ}}
\newcommand{\occ} {\ensuremath{\bullet}}

\newcommand{\be}{\begin{equation}}
\newcommand{\ee}{\end{equation}}
\newcommand{\bew}{\begin{equation*}}
\newcommand{\eew}{\end{equation*}}
\newcommand{\mlq}{multiline queue}
\newcommand{\mlqs}{multiline queues}

\newcommand{\refT}[1]{Theorem~\ref{#1}}

\newcommand{\refL}[1]{Lemma~\ref{#1}}

\newcommand{\refS}[1]{Section~\ref{#1}}
\newcommand{\refF}[1]{Figure~\ref{#1}}

\newenvironment{romenumerate}[1][0pt]{
\addtolength{\leftmargini}{#1}\begin{enumerate}
 }{\end{enumerate}}

\newtheorem{theorem}{Theorem}[section]
\newtheorem{lemma}[theorem]{Lemma}
\newtheorem{proposition}[theorem]{Proposition}

\newtheorem{algorithm}[theorem]{Algorithm}
\newtheorem{definition}[theorem]{Definition}
\newtheorem{remark}[theorem]{Remark}

\newtheorem{conjecture}[theorem]{Conjecture}

\begin{document}

\title[An Inhomogeneous Multispecies TASEP on a Ring]
{An Inhomogeneous Multispecies \\ TASEP on a Ring}

\author{Arvind Ayyer} 
\address[Arvind Ayyer]{Department of Mathematics, UC Davis, One Shields Ave., Davis, CA 95616-8633, U.S.A. \newline
New address: Department of Mathematics, Department of Mathematics, Indian Institute of Science, Bangalore - 560012, India.
}
\email{arvind@math.iisc.ernet.in}

\author{Svante Linusson} 
\thanks{Svante Linusson is a Royal Swedish Academy of Sciences Research Fellow supported by a grant from the Knut and Alice Wallenberg Foundation.}
\address[Svante Linusson]{Svante Linusson, Department of Mathematics \\ KTH-Royal Institute of Technology, 
  SE-100 44, Stockholm, Sweden.}
\email{linusson@math.kth.se}

\date{\today}  

\begin{abstract}
We reinterpret and generalize conjectures of Lam and Will\-iams  
as statements about the stationary distribution of a multispecies 
exclusion process on the ring. 
The central objects in our study are the multiline queues of
Ferrari and Martin.
We make some progress on some of the conjectures in different directions. 
First, we prove Lam and Williams' conjectures in two special cases by generalizing the
rates of the Ferrari-Martin transitions. Secondly, we define a new process on multiline queues, which have a certain minimality 
property. This gives another proof for one of the special cases; namely arbitrary jump rates for three species.
\end{abstract}

\maketitle

\section{Introduction} \label{S:Intro}

We study a totally asymmetric simple exclusion process (TASEP in short)  on the ring $\mathbb Z/N\mathbb Z$, which is a continuous time Markov chain in which each  position of the ring is occupied by exactly one particle of a certain class. The dynamical moves in the TASEP are that a particle can jump over (i.e. trade places with) a particle from a larger (i.e., higher numbered) class to its left. TASEPs have been extensively studied in the physics literature. We give some background on the TASEP in Section~\ref{S:backg}.

Lam \cite{lam} conjectured properties of this Markov chain which he needed for his work on infinite reduced words of affine Weyl groups. This TASEP has been studied independently by several authors and in particular it has been given a beautiful solution in terms of so-called multiline queues by Ferrari and Martin \cite{FM2}, which we describe in \refS{S:Multi}. These intricate objects give a solution to, and predate, one of Lam's conjectures about the partition function, but more work is needed to resolve others, in particular the stationary weight for the identity.

In \cite{lamwill} Lam and Williams generalized the model, but in a different language, by including different jump rates for different classes of particles, so particles of class $i$ jump with rate $x_i$. They obtained surprisingly nice stationary weights, which they conjectured to be polynomials in the jump rates with positive integer coefficients, see Conjecture \ref{C:LW}. The purpose of the present paper is to report some advances on these conjectures.

In particular, see \refS{S:FM3}, we solve this inhomogeneous TASEP with three different 
kinds of particles, \refT{T:FM3}, verifying this conjecture. In fact we do this in two different ways, first using the transitions for multiline queues by Ferrari and Martin and then by defining a new different Markov chain, see 
\refS{S:new}, which also projects down and gives a solution of the TASEP with three different classes of particles. We include this second proof because we have higher hopes that it may generalize to $n\ge 4$. We also prove in \refS{S:FM1} that the Markov chain
defined by Ferrari and Martin is enough to understand the 
power of $x_1$, setting $x_2=\dots=x_{n-1}$, for an arbitrary system with exactly one first class particle. The interested reader can verify the statements in the paper by downloading the Maple package \texttt{InhomTasep.maple} from either the homepage of one of the authors or the \texttt{arXiv} source. 

{\em Note added in proof:} Since the submission of the paper there has been progress and there are now claimed proofs of both the formula for the stationary weight of the identity by Aas \cite{Aas} and for our Conjecture \ref{C:main} (two very different ways by Arita-Malik \cite{AM} and Linusson-Martin \cite{LM}).

\subsection{Background on TASEP} \label{S:backg}
The general exclusion process can be defined on an arbitrary graph. One starts with a configuration of particles on the vertices of the graph, where every vertex can be occupied by at most one particle. The process involved hopping of the particles according to a Markovian rule, and the exclusion condition refers to the fact that at most one particle can be at any site.

The exclusion process was probably first studied in the biological literature. It was proposed as a prototype to describe the dynamics of ribosomes along RNA \cite{MGP}. Exclusion processes were studied systematically by probabilists in the 70s starting with the work of Spitzer \cite{Spitzer}, who coined the term. A lot of results are now known about the exclusion process on $\mathbb Z$ \cite{Lig2}.

There were two starting points for a more
 combinatorial understanding of exclusion processes on a finite state space, one for the TASEP \cite{DS} by Duchi and Schaeffer, and one for the PASEP \cite{CW1} by Corteel and Williams.
 There are many subsequent papers relating these to well-known combinatorial constructs. Recently, this approach has led to the first combinatorial interpretation of the moments of the Askey-Wilson polynomials \cite{CW2}.

Motivated by questions in statistical physics, Derrida, Janowsky, Lebowitz and Speer \cite{DJLS} obtained the stationary distribution for the TASEP on
$\mathbb Z/N\mathbb Z$ with two species of particles, first-class and second-class, in addition to vacancies using the {\em Matrix Ansatz} technique. 
The general solution of the stationary distribution for any number of different classes of particles  came from Ferrari and Martin \cite{FM2}, who built on their own work on multiline queues \cite{FM1} and used ideas from  Ferrari, Fontes and Kohayakawa \cite{FFK} and Angel \cite{A}. Building on this work, the matrix ansatz solution for the general TASEP
was constructed in \cite{EFM}, and for the general PASEP in \cite{PEM}. For a general review of the matrix ansatz, see \cite{BE}. We note that the inhomogeneous TASEP has also been studied in relation to models of vehicular traffic, see \cite{LPK}.

The one-dimensional totally asymmetric exclusion process arose again, using very different terminology, in the study of random walks in Weyl alcoves by Lam \cite{lam}. He managed to prove that some results about the infinite Markov chain could be obtained by studying the stationary distribution of a finite one on permutations. This finite Markov chain turns out to be equivalent to the multispecies TASEP on a ring discussed above with one particle of each class. Further work along this direction led to more conjectures by Lam and Williams \cite{lamwill}, in particular the positivity conjecture addressed in this paper.

\section*{Acknowledgements}
We thank Thomas Lam and Lauren Williams  for valuable discussions. 
The authors would like to thank the MSRI, where this research was conducted, for their support and hospitality. 
We also thank Mireille  Bousquet-M\'elou and an anonymous referee for comments on the manuscript.

\section{Multispecies Exclusion Processes and Multiline Queues}\label{S:Multi}
The aim of this section is to define the quantities in the title, and explain
how they relate to one another.

Before we define these processes, we quickly recall the part of the theory of Markov chains relevant to us.
The continuous-time Markov chain on a (finite) state space $\om$ is defined by the so-called {\bf transition matrix} or $M$, whose rows and columns are labeled by elements of $\om$. 
For $\tau_{1},\tau_{2} \in \om$, we take the convention that the $(\tau_{1},\tau_{2})$ entry is the
rate of the transition from  $\tau_{2} \to \tau_{1}$ if $\tau_{1} \neq \tau_{2}$. 
The diagonal entries $(\tau_{1},\tau_{1})$ is the negative of the sum of the rates of 
transitions leaving $\tau_{1}$, not counting loops. 
\bew
(M)_{\tau_{1},\tau_{2}} = \begin{cases}
\text{rate}(\tau_{2} \to \tau_{1}), & \tau_{1} \neq \tau_{2} \\
\ds -\sum_{\tau \in \om\setminus\{\tau_1\}} \text{rate}(\tau_{1} \to \tau), & \tau_{1} = \tau_{2}.
\end{cases}
\eew

This ensures that column sums are zero and consequently, zero is an eigenvalue with row (left-) eigenvector being the all-ones vector.
Assuming that the Markov chain is {\bf irreducible}, it has a unique {\bf stationary distribution} 
\cite{levperwil}. 
This is given by the entries of the corresponding column (right-) eigenvector $w$ with eigenvalue zero.

The stationary weight vector $w$ is thus determined by solving the equation $M.w=0$. This can be written in the following way. Let $w(\tau)$ be the stationary distribution of configuration $\tau$, which can be viewed as the entry of the vector $w$ at position $\tau$. Then
\be \label{mastereq}
\sum_{\tau' \in \om\setminus\{\tau\}} \text{rate}(\tau' \to \tau) w(\tau') = \sum_{\tau' \in \om\setminus\{\tau\}} \text{rate}(\tau \to \tau') w(\tau).
\ee

This is the all-important {\bf master equation} for $\tau$. We will use the uniqueness (up to an overall scaling factor) of the solution of these 
equations to prove all our results about stationary distributions. For an incoming transition from $\tau' \to \tau$, we will use the term {\bf effective rate} 
to mean 
\be \label{effrate}
\text{rate}(\tau' \to \tau) \frac{w(\tau')}{w(\tau)}.
\ee 
The master equation \eqref{mastereq} is then equivalent to that the outgoing rates sum up to the same thing as the incoming effective rates.

\subsection{Multitype TASEP on a ring}
The state space of the $n$-species exclusion process $\om_{m}$ is defined by an 
$n$-tuple $m = (m_{1},\dots,m_{n})$ of positive integers where $m_{1}+\cdots+m_{n}=N$.
The configurations $\pi \in \{1,\dots,n\}^{N}$ are those $N$-tuples with $m_{1}$ number of 1's, $m_{2}$ number of 2's, and so on. More precisely,
\be \label{TASEPconf}
\om_{(m_{1},\dots,m_{n})} = \{(\pi_{1},\dots,\pi_{N}) \;\; | \;\;
\# \{i| \pi_{i} = j \} = m_{j} \text{ for $j = 1,\dots,n$.} 
\}
\ee
Clearly,
\bew
|\om_{m}| = \binom N{m_{1},\dots,m_{n}} \nonumber.
\eew
$\pi_{i}$ denotes the  {\bf class} or equivalently, the {\bf species} of the particle at the $i$th site. We will use the terms `class' and `species' interchangeably.
Each species can be thought of as a positive integer, which wants to go to the left  and lower integers take precedence over higher ones. This might seem counterintuitive, but this notation makes sense if we think of first-class particles (of class 1) being superior to second-class particles (of class 2). 

\begin{definition}
The {\bf multispecies exclusion process} on $\om_{m}$ is defined by local transitions
involving sites $i$ and $i+1 \pmod N$. If the current state is $\pi$, with $\pi_{i}=a$, 
$\pi_{i+1}=b$, then $a$ and $b$ can interchange positions, namely
\be \label{rates-hom}
a\;b \to b\;a \text{ with rate 1 if } a>b.\\
\ee
\end{definition}

In words, each particle carries an exponential clock which rings with rate 1, and the particle tries to jump to its left whenever the clock rings. 
If the particle to the left is of less importance the jump takes place.
Otherwise, the configuration is unchanged. 

Therefore 1's always move left and $n$'s always move right. Particles whose class is between 
1 and $n$  sometimes move left, and sometimes right.
We consider this to be a homogeneous model because the rate of all transitions is 1. The model of Lam and Williams \cite{lamwill} on permutations, 
on the other hand, is inhomogeneous because transition rates depend on the particles being interchanged. 
The careful reader will notice that our Markov chain is different from theirs 
because we interchange neighboring positions whereas they interchange neighboring values. In the theory of permutation groups, this is just the difference between multiplying with simple transpositions on the right or on the left, so there is an easy bijection between the permutations in two 
chains. In particular, the (multi-)set of stationary weights are given by the same expressions.

We now generalize their model to multipermutations, where the rate of the exponential clock attached to particle of species $j$ is $x_{j}$. They are to be thought of as positive numbers, but we treat them as formal parameters.

\begin{definition}\label{D:trans}
The {\bf inhomogeneous multispecies exclusion process} on $\om_{m}$ is defined by local 
transitions involving sites $i$ and $i+1 \pmod N$. If the current state is $\pi$, 
with $\pi_{i}=a$, $\pi_{i+1}=b$, then $a$ and $b$ can interchange positions, namely
\be \label{rates}
a\;b \to b\;a \text{ with rate $x_{b}$ if } a>b.\\
\ee
\end{definition}

We give the $m=(1,1,1)$ example in detail now. The graph of the Markov chain and the stationary weights of each configuration are given in 
Figure~\ref{F:permeg}.
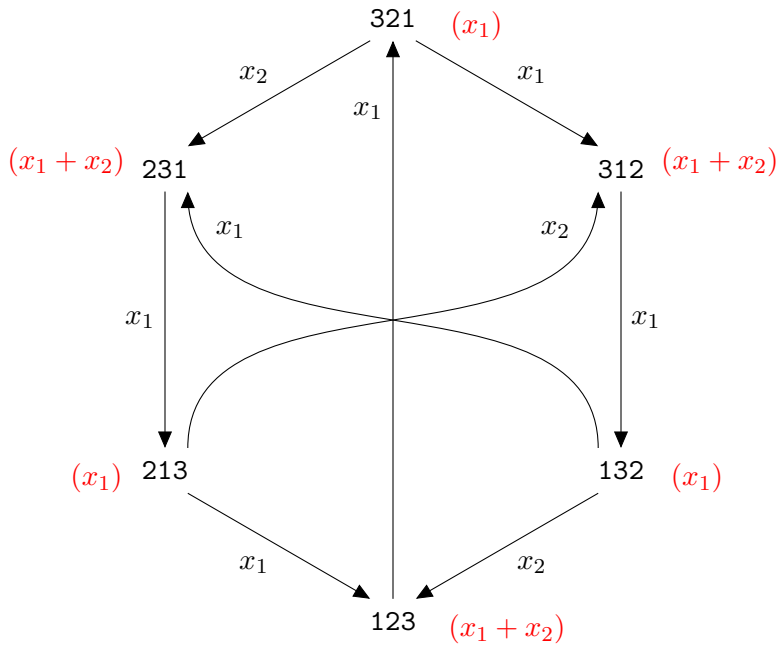
\begin{figure}[ht]
\begin{center}
\begin{tikzpicture} [>=triangle 45]
\draw (3,8) node {\verb!321!};
\draw (4.1,7.9) node [color=red] {$(x_{1})$};
\draw (0,6) node {\verb!231!};
\draw (-1.3,6.1) node [color=red] {$(x_{1}+x_{2})$};
\draw (6,6) node {\verb!312!};
\draw (7.3,6.1) node [color=red] {$(x_{1}+x_{2})$};
\draw (0,2) node {\verb!213!};
\draw (-0.9,1.9) node [color=red] {$(x_{1})$};
\draw (6,2) node {\verb!132!};
\draw (7.0,1.9) node [color=red] {$(x_{1})$};
\draw (3,0) node {\verb!123!};
\draw (4.5,-0.1) node [color=red] {$(x_{1}+x_{2})$};
\draw [->] (2.7,7.7) -- node [auto,swap] {$x_{2}$} (0.3,6.3);
\draw [->] (3.3,7.7) -- node [auto] {$x_{1}$} (5.7,6.3);
\draw [->] (0,5.7) -- node [auto,swap] {$x_{1}$} (0,2.3);
\draw [<-] (2.7,0.3) -- node [auto] {$x_{1}$} (0.3,1.7);
\draw [<-] (3.3,0.3) -- node [auto,swap] {$x_{2}$} (5.7,1.7);
\draw [<-] (6,2.3) -- node [auto,swap] {$x_{1}$} (6,5.7);
\draw [->] (3,0.3) -- node [auto,very near end] {$x_{1}$} (3,7.7);
\draw [->] (5.7,2.3) to [out = 90, in = -90] node [auto,very near end,swap] {$x_{1}$}(0.3,5.7);
\draw [->] (0.3,2.3) to [out = 90, in = -90] node [auto,very near end] {$x_{2}$}(5.7,5.7);
\end{tikzpicture}
\end{center}
\caption{Transitions of the Markov chain for $n=3$. The stationary weights are given in parenthesis and in red.}
\label{F:permeg}
\end{figure}
The transition matrix  in the lexicographically ordered basis, 
$\{123, 132, 213, 231, 312, 321 \}$ is given by
\bew
\begin{pmatrix}
-x_{1} & x_{2} & x_{1} & 0 & 0 & 0 \\
0 & -x_{1}-x_{2} & 0 & 0 & x_{1} & 0 \\
0 & 0 & -x_{1}-x_{2} & x_{1} & 0 & 0 \\
0 & x_{1} & 0 & -x_{1} & 0 & x_{2} \\
0 & 0 & x_{2} & 0 & -x_{1} & x_{1} \\
x_{1} & 0 & 0 & 0 & 0 & -x_{1}-x_{2}
\end{pmatrix}
\eew
Notice that permutations which are obtained by rotations of one another have the same stationary
weights. This is true in general because the transition rates defined in \eqref{rates} depend only on the particle classes, not on their positions.

\subsection{Ferrari-Martin multiline queues}
We now define the multiline que\-ues of Ferrari and Martin \cite{FM2}.
As before, let $m$ be the type of the process, with $n$ species and $N$ sites. 
Then the state space of the process, $\ofm_{m}$ is defined on a cylinder 
of circumference $N$ and height $n-1$, each site of which is either occupied by a $\occ$ or a $\vac$.
We think of $\occ$ as being occupied and $\vac$ as being vacant.
It will be convenient to set $M_{r} = \sum_{i=1}^{r} m_{i}$. 
At row $r$, for $1 \leq r <n$, the configuration will contain 
$M_{r}$ $\occ$'s and $N-M_{r}$ $\vac$'s. The rows are numbered from the top and the indices are given with 
row first, column second. 
At each row, all possible configurations with the prescribed number of $\occ$'s and $\vac$'s are allowed.
\be \label{queueconf}
\ofm_{(m_{1},\dots,m_{n})} = \left\{(\con_{r,i})_{1,1}^{n-1,N} \;\; \Big| \;\;
\substack{\ds \con_{r,i} \in \{\vac,\occ\},  \\
\\
\ds \#\{i| \con_{r,i} = \occ \} = M_{r} \text{ for $r \in \{1,\dots,n-1\}$.}}
\right\}
\ee
Since configurations at different rows can be chosen independently,
\[
|\ofm_{m}| = \prod_{r=1}^{n-1} \binom N{M_{r}}.
\]

To describe the transitions of this new Markov chain, we need a notion defined
in \cite{FM1}, which we will give a new name. 

\begin{definition} \label{D:ringingpath}
Given a configuration $\con \in \ofm_{m}$, the {\bf ringing path} 
$P^{(i)}=(P^{(i)}_1,\dots ,P^{(i)}_{n-1})$ is defined as follows: 
$P^{(i)}_{n-1}=i$ and 
\bew
P^{(i)}_{r-1} = \begin{cases}
P^{(i)}_{r}, & \text{if }\con_{r,P^{(i)}_{r}} = \occ, \\
P^{(i)}_{r}+1, \pmod N & \text{if } \con_{r,P^{(i)}_{r}} = \vac.
\end{cases}
\eew
There is a possible transition between $P^{(i)}_{r}-1$ and $P^{(i)}_{r}$ simultaneously for every row $r$. If we set $a=\con_{r,P^{(i)}_{r}-1}$ and $b=\con_{r,P^{(i)}_{r}}$, then $a\;b \to b\;a$ if and only if $a=\vac$ and $b=\occ$.
This is called a {\bf ringing path transition} at position $i$. 
\end{definition}

We illustrate this idea with an example in $\ofm_{(1,1,2,2,2)}$ in Figure~\ref{F:clockseq}
(ignore the numbers at the bottom for the moment).
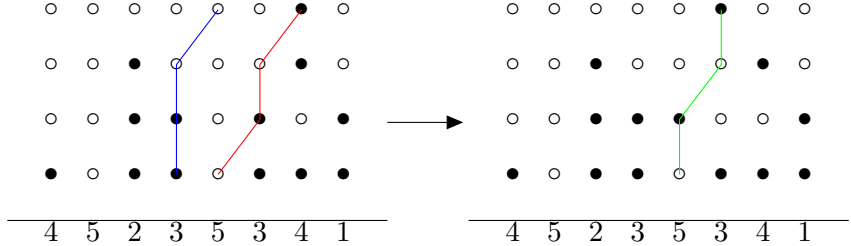
\begin{figure}[h!]
\begin{center}
\begin{tikzpicture}[>=triangle 45]
\matrix [column sep=0.1cm, row sep = 0.3cm] 
{
\node {$\vac$}; & \node {$\vac$}; & \node {$\vac$}; & \node {$\vac$}; & 
\node(b4) {$\vac$}; & \node {$\vac$}; & \node(a4){$\occ$}; & \node {$\vac$}; \\ 
\node {$\vac$}; & \node {$\vac$}; & \node {$\occ$}; & \node(b3) {$\vac$}; & 
\node {$\vac$}; & \node(a3) {$\vac$}; & \node {$\occ$}; & \node {$\vac$}; \\ 
\node {$\vac$}; & \node {$\vac$}; & \node {$\occ$}; & \node(b2) {$\occ$}; & 
\node {$\vac$}; & \node(a2) {$\occ$}; & \node {$\vac$}; & \node {$\occ$}; \\ 
\node {$\occ$}; & \node {$\vac$}; & \node {$\occ$}; & \node(b1) {$\occ$}; & 
\node(a1) {$\vac$}; & \node {$\occ$}; & \node {$\occ$}; & \node {$\occ$}; \\ 
\node {$4$}; & \node {$5$}; & \node {$2$}; & \node {$3$}; & 
\node {$5$}; & \node {$3$}; & \node {$4$}; & \node {$1$}; \\ 
}; 
\draw (-2.5,-1.3) -- (2.5,-1.3);
\draw [-,red] (a1.center) --  (a2.center) -- (a3.center) -- (a4.center);
\draw [-,blue] (b1.center) --  (b2.center) -- (b3.center) -- (b4.center);
\draw [->] (2.5,0) -- (3.5,0);
\end{tikzpicture}
\hspace{-0.2cm}
\begin{tikzpicture}[>=triangle 45]
\matrix [column sep=0.1cm, row sep = 0.3cm] 
{
\node {$\vac$}; & \node {$\vac$}; & \node {$\vac$}; & \node {$\vac$}; & 
\node {$\vac$}; & \node (c4){$\occ$}; & \node {$\vac$}; & \node {$\vac$}; \\ 
\node {$\vac$}; & \node {$\vac$}; & \node {$\occ$}; & \node {$\vac$}; & 
\node {$\vac$}; & \node (c3){$\vac$}; & \node {$\occ$}; & \node {$\vac$}; \\ 
\node {$\vac$}; & \node {$\vac$}; & \node {$\occ$}; & \node {$\occ$}; & 
\node (c2){$\occ$}; & \node {$\vac$}; & \node {$\vac$}; & \node {$\occ$}; \\ 
\node {$\occ$}; & \node {$\vac$}; & \node {$\occ$}; & \node {$\occ$}; & 
\node (c1){$\vac$}; & \node {$\occ$}; & \node {$\occ$}; & \node {$\occ$}; \\ 
\node {$4$}; & \node {$5$}; & \node {$2$}; & \node {$3$}; & 
\node {$5$}; & \node {$3$}; & \node {$4$}; & \node {$1$}; \\ 
}; 
\draw (-2.5,-1.3) -- (2.5,-1.3);
\draw [-,green] (c1.center) --  (c2.center) -- (c3.center) -- (c4.center);
\end{tikzpicture}
\caption{The clock rings at site 5, causing the red ringing path transition $P^{(5)} = (7,6,6,5)$, causing the resulting green path. If the clock rang at site 4 instead, then the blue ringing path $P^{(4)}=(6,5,4,4)$ would cause no transition.}
\label{F:clockseq}
\end{center}
\end{figure}

We are now in a position to define the Markov chain on the queues of \eqref{queueconf}.
\begin{definition}
The {\bf Ferrari-Martin multiline process} is the Markov chain on $\ofm_{m}$ where the dynamics occurs by ringing path transitions that take place via exponential clocks of rate 1 at all sites on row $n-1$. 
\end{definition}

In the following we summarize some facts about the Markov chain $\ofm_m$ proved in \cite [Prop 3.3, Theorem 3.1, Prop 3.2]{FM2}.

\begin{theorem}[Ferrari and Martin, \cite{FM2}] 
\label{T:FMringing}
\hspace*{1cm}

\begin{enumerate}
\item The ringing path transitions have an inverse given by starting a ringing path similar to Definition~\ref{D:ringingpath} at each position starting at the first row and moving in the opposite direction. This implies in particular that the number of incoming and outgoing transitions is the same for every state in $\ofm_m$.

\item One can get from any multiline queue to any other multiline queue using the ringing path transitions.

\item The (unique) stationary distribution of the Ferrari-Martin multiline process is the uniform distribution.

\item Each row of the multiline queue is a standard exclusion process with $n=2$, where $\occ$'s behave like 1s and $\vac$'s behave like 2s.
\end{enumerate}
\end{theorem}

\subsection{Bully-path projection}
There is a formal notion of projection of Markov chains called {\bf lumpability} \cite[Lemma 2.5]{levperwil}, which we state for completeness. Let $\om$ be the state space of a Markov chain with transition matrix $M$. If $\om$ can be partitioned into equivalence classes, denoted $[\cdot]$, so that
\(
M(x,[y]) = M(x',[y])
\)
whenever $x' \in [x]$, then the resulting process on the equivalence classes is also a Markov chain. Then $\om$ is said to be lumpable with respect to the equivalence relation. 

\begin{algorithm} \label{A:bullypath}
The {\bf bully-path projection} is a map $\bp :\ofm_{m} \to \om_{m}$.
 Given a multiline queue $\con \in \ofm_{m}$ we will assign a class to all $\occ$'s in $\con$ inductively by the following procedure.

\begin{enumerate}
\item All the $\occ$'s in the first row ($r=1$) are defined to be particles of class 1.
\item Assume we have classified $m_k$ $\occ$'s as class $k$ for $1\le k\le r$ on row $r$ and want to do the same for row $r+1$. 

\item First take a $\occ$ of class 1 on position $(r,i)$ and look at $(r+1,i)$, the site below. If it is occupied, then that $\occ$ is defined to be of class 1. 
If $(r+1,i)$ is vacant, let $(r+1,i')$  be the first non-vacant position to the right$\pmod N$. The $\occ$ on that position is then
classified as first class. We say that the $\occ$ (the
particle) is {\bf queueing} during the positions $(r+1,i), (r+1,i+1),\dots (r+1,i')$ until it encounters an occupied site (which is called a service 
opportunity in \cite{FM2}). We also think pictorially of a path from $(r,i)$ down to 
$(r+1,i)$ and then rightwards all the way to $(r+1,i')$. 

\item The same process is repeated with all other $\occ$'s of class 1 with the additional rule that already classified 
$\occ$'s are ignored, i.e. the particle is queueing until it finds a non-classified $\occ$. The same procedure is then repeated for all 
$\occ$'s on row $r$ making sure that we assign first all of class $k$ before those of class $k+1$. 

\item When we have done this for all $\occ$'s on row 
$r$ there will be $m_{r+1}>0$ non-classified occupied positions left on row $r+1$ and they are all given class $r+1$. 

\item This procedure is continued to the bottom row ($r=n-1$) where a state of $\om_m$ is obtained if we assign the vacant positions as particles of class $n$.
\end{enumerate}
\end{algorithm}

\begin{theorem}[Ferrari and Martin, \cite{FM2}] \label{T:FMbully} \hspace{1cm}
\begin{enumerate}
\item The bully-path projection procedure in Algorithm~\ref{A:bullypath} is well-defined in the sense that the outcome does not depend on the order in which the particles within a class are classified.

\item The bully-path projection lumps the Ferrari-Martin multiline process on $\ofm_m$
to the  multispecies exclusion process on $\om_{m}$. Moreover, 
if we set $m'=(m_{1},\dots,m_{r}, N-M_r)$,
the bully-path projection up until row $r$ results in the lumping of the Ferrari-Martin multiline process on $\ofm_{m'}$ to the multispecies exclusion process $\om_{m'}$.

\item If the ringing path transition at position $i$ of $\con\in \ofm_m$ results in $\con'$, then $\bp(\con')$ is obtained from 
$\bp(\con)$ by the $\om_m$-transition at sites $i-1$ and $i$ as in Definition \ref{D:trans}.
\end{enumerate}
\end{theorem} 

For example, one can check that both the multiline queues in Figure~\ref{F:clockseq} project to numbers shown, irrespective of the order of choosing the $\occ$'s. Moreover, the order will not affect the class of any $\occ$ at any row.

For each particle we can join the paths in the recursive definition to a trace which we call its {\bf bully path}. 
Hence, a bully path for a class $j$ particle starts where the particle first is classified on row $j$ and moves down one row and then 
to the right (circularly) on a row while the particle is queueing on that row. Then the path continues 
downwards from the new classified $\occ$ and continues recursively. Note that the bully paths are not uniquely defined, but the classes of the $\occ$'s are. 
An example of the bully-path projection is given in Figure~\ref{F:bullypatheg}.

\begin{figure}[h!]
\begin{center}
\begin{tikzpicture}
\matrix [column sep=1cm, row sep = 0.5cm] 
{
\node {$\vac$}; & \node{$\vac$}; & \node(a1) {$\occ_1$}; & \node {$\vac$}; & \node {$\vac$}; & \node {$\vac$}; \\ 
\node {$\vac$}; & \node(b1){$\occ_2$}; & \node(a2) {$\occ_1$}; & \node {$\vac$}; & \node {$\vac$}; & \node {$\vac$}; \\ 
\node(c1) {$\occ_3$}; & \node(b2){$\vac$}; & \node(a3) {$\vac$}; & \node(a4) {$\vac$}; & \node(a5) {$\occ_1$}; & \node(b3) {$\occ_2$}; \\ 
\node(b5) {$\occ_2$}; & \node(c3){$\occ_3$}; & \node {$\vac$}; & \node(d1) {$\occ_4$}; & \node(a6) {$\vac$}; & \node(a7) {$\occ_1$}; \\ 
\node(a9) {$\occ_1$}; & \node(b7){$\occ_2$}; & \node(c5) {$\occ_3$}; & \node(d3) {$\occ_4$}; & \node(e1) {$\occ_5$}; & \node(a8) {$\vac$}; \\ 
\node(a10) {1}; & \node(b8){2}; & \node(c6) {3}; & \node(d5) {4}; & \node(e3) {5}; & \node {6}; \\ 
}; 
\draw [-,red] (a1.center) -- (a2.center) -- (a3.center) -- (a4.center) -- (a5.center) -- (a6.center) -- (a7.center) -- (a8.center) -- (4.3,-1.45);
\draw [-,red] (-4.5,-1.45) -- (a9.center) -- (a10.center);
\draw [-,blue] (b1.center) -- (-2.4,0.65) -- (4,0.65) -- (4,-0.3) -- (4.3,-0.3);
\draw [-,blue] (-4.5,-0.5) -- (b5.center)  -- (-4,-1.25) -- (-2.45,-1.25) -- (b8.center);
\draw [-,green] (c1.center) -- (-4,-0.3) -- (-2.4,-0.3) -- (-2.4,-1.2) -- (-0.8,-1.2) -- (c6.center);
\draw [-,yellow] (d1.center) -- (d3.center) -- (d5.center);
\draw [-,orange] (e1.center) -- (e3.center);
\end{tikzpicture}
\caption{An example of a multiline queue with its set of bully paths for $N=n=6$.
The bully-path projection defined in Algorithm~\ref{A:bullypath} is written below.}
\label{F:bullypatheg}
\end{center}
\end{figure}
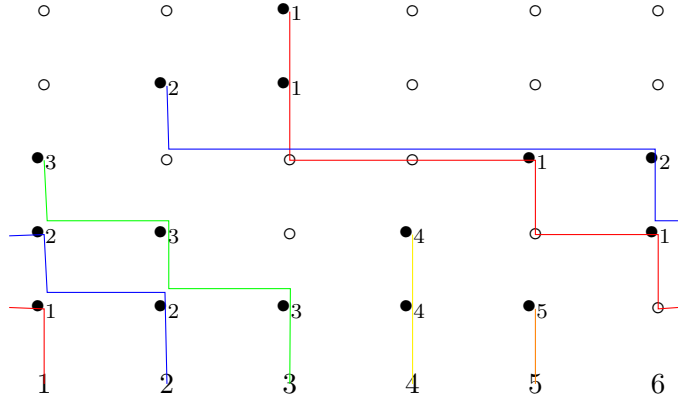

\smallskip
We concentrate on following conjecture of Lam and Williams \cite{lamwill}, where $N=n$.
Let  $w_0$ denote the reverse permutation $n,n-1,\dots,2,1$. 

\begin{conjecture} (Lam and Williams \cite{lamwill})\label{C:LW}
With the normalization $w(w_0)=x_1^{\binom{n-1}{2}}x_2^{\binom{n-2}{2}}\dots x_{n-2}^{\binom{2}{2}}$ every stationary weight
$w(\pi)$ is a polynomial with positive integer coefficients.
\end{conjecture}

From \refT{T:FMringing}, we know that the stationary distribution on multiline queues is uniform when all the $x_{i}$'s are equal to 1. This means that to find the stationary weight of a given
permutation it suffices to count the number of multiline queues that project to that permutation. Let us point out that there is no known easy way 
of doing this. In particular, the conjecture by Lam \cite{lam} about the stationary weight of the identity is not proved. It would amount to proving that
there are $\prod_{i=1}^{n-1}\binom{n-1}{i}$ \mlqs{} that project to the identity.

Our aim is to define a Markov chain on the multiline queues such that the stationary weight in the general setting is a monomial, and the sum of 
the monomials corresponding to a certain multipermutation under the projection $\bp$ leads to the stationary distribution of the 
multispecies exclusion process
with the transition rules given by \eqref{rates}. This would give a combinatorial proof of Conjecture \ref{C:LW}, and moreover a generalization of that conjecture to multipermutations. 

We have not settled their conjecture in full generality, but we do have partial results which will be presented in the remaining three sections. 
We now present a conjecture describing the stationary weight of each \mlq.  
Let $\con$ be a configuration of length $N$ with $n$ different classes of particles. A vacancy  $\vac$ in $\con$ is called 
$i${\bf -covered} if it is  traversed by a bully path from an $i$-class particle (i.e. starting on row $i$) , but not traversed by any path from an
 $i'$-class particle, with $i'<i$. 
 For $1\le i<r\le n$, let $z_{r,i}(\con)$ be the number of $\vac$'s on row $r$ of $\con$  that are $i$-covered. 
 
\begin{remark}
 We note that $z_{r,i}$ is well-defined in the sense that it does not depend on the order of bully paths, see discussion before 
 Conjecture \ref{C:LW}.
\end{remark}
  
Let $v_r:=N-M_r$ be the number of vacancies $\vac$ on row $r$, and let $V_r:=\sum_{i=r+1}^{n-1} v_i$ be the
number of vacancies below row $r$. Our main conjecture is as follows.

\begin{conjecture}\label{C:main}
There is a Markov chain on the state space of $\ofm_m$ which lumps via the bully-path projection in Algorithm~\ref{A:bullypath} to the inhomogeneous multispecies exclusion process on $\om_{m}$ such that the stationary weight of any configuration $\con$ is given by
\[
w(C) = x_1^{V_1}x_2^{V_2}\dots x_{n-2}^{V_{n-2}}
\prod_{1\le i<r\le n} \left( \frac{x_r}{x_i} \right)^{z_{r,i}}.
\]
\end{conjecture}

Our results in \refS{S:FM3} and \refS{S:FM1} support this conjecture. In these sections, 
we use  the ringing path  transitions of Ferrari-Martin. 
It seems difficult to extend this to approach by giving weights to the ringing path transitions when $n\ge 4$. We believe that it is not possible to do so if we restrict to a monomial as weight for each \mlq{} and a single $x_i$ for each transition.

\noindent
{\bf Example} Consider the \mlq{} $\con$ in \refF{F:bullypatheg}. Since it is a permutation we have $V_r=\binom{n-r}{2}$, for all $r$.
The only non-zero $z_{i,j}$'s are $z_{1,3}=2, z_{1,4}=z_{1,5}=z_{2,3}=1$. Conjecture \ref{C:main} states that
\[w(\con)=x_1^{10}x_2^6x_3^3x_4\left(\frac{x_3}{x_1}\right)^{2}\left(\frac{x_4}{x_1}\right)\left(\frac{x_5}{x_1}\right)\left(\frac{x_3}{x_2}\right)=
x_1^{6}x_2^5x_3^6x_4^2x_5.
\]

We have checked these conjectures for all possible configurations up to size $N=6$ on Maple$^{TM}$.
We did these by comparing two quantities. First, we calculated the stationary  weight for each multipermutation by looking at the inhomogeneous multispecies exclusion process. Secondly, we calculated the weight for each configuration by adding the stationary weight as given by Conjecture \ref{C:main}  of each \mlq{} which projected to it.

\section{Three species Ferrari-Martin process} \label{S:FM3}
Our first result is a proof of the natural generalization of  Conjecture \ref{C:LW} for multipermutations with three classes, i.e. $n=3$. To state the result, we will simplify notation in the following way. 
Let $\con$ be a configuration in $\ofm_{(m_{1},m_{2},m_{3})}$ whose bully-path projection 
$\bp\con= \pi$ with $\pi_{i}=3$. 
We say that the 3 is {\bf covered} 
if the $\vac$ on row 2 and column $i$ is 1-covered, i.e. if there is a bully path passing through that $\vac$.
If no such bully path exists, we say that the 3 is {\bf non-covered}.
We will write $\bp\con_i$ for $\bp(\con)_i$ from now on.

\begin{definition} \label{D:FMq3}
The {\bf inhomogeneous Ferrari-Martin multiline process on three species} 
is the Markov chain on $\ofm_{(m_{1},m_{2},m_{3})}$ where the dynamics occurs is governed 
by ringing path transitions as follows. 
If the transition is from $\con$ to $\con'$  and the ringing path starts at $i$, then 
\[
\text{rate}(\con \to \con') = \begin{cases}
x_{1} & \bp\con_{i} = 1 \text{ or } \bp\con_{i} = \text{ covered } 3, \\
x_{2} & \bp\con_{i} = 2 \text{ or } \bp\con_{i} = \text{ non-covered } 3.
\end{cases}
\]
\end{definition}

The graph of the Markov chain for the multiline queue $m=(1,1,1)$ is given in \refF{F:mlq1eg}.
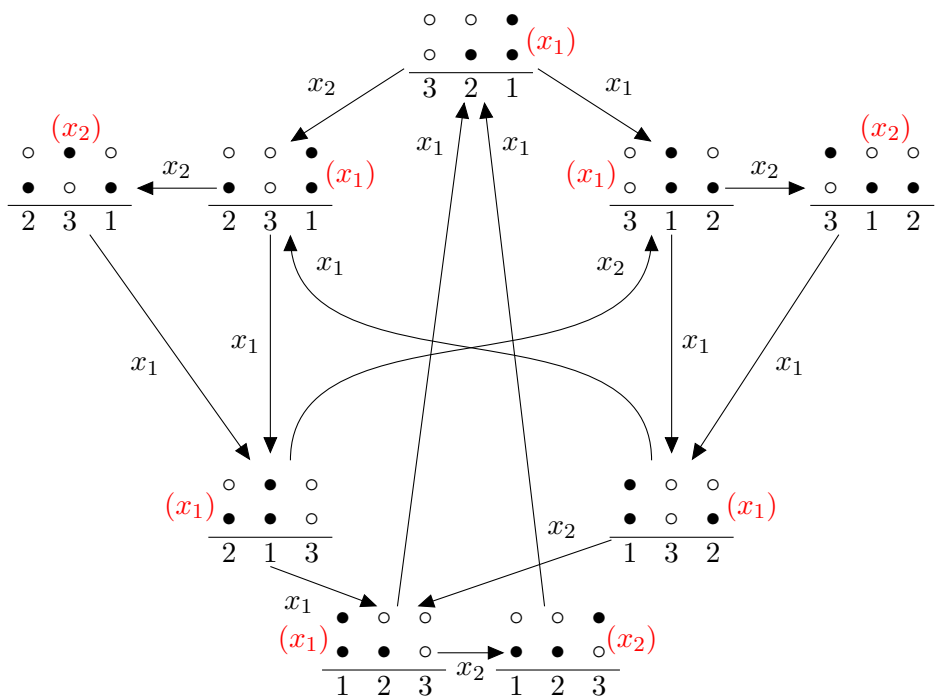
\begin{figure}[h]
\begin{center}
\begin{tikzpicture} [scale=0.88, >=triangle 45]
\draw (3,9) node 
{$
\begin{array}{c c c} \vac & \vac & \occ \\ \vac & \occ & \occ\\ \hline 3 & 2 & 1 \end{array}$
};
\draw (4.2,9.2) node [color=red] {$(x_{1})$};
\draw (6,7) node 
{
$\begin{array}{c c c} \vac & \occ & \vac  \\ \vac & \occ & \occ\\ \hline 3 & 1 & 2 \end{array}$
};
\draw (4.8,7.2) node [color=red] {$(x_{1})$};
\draw (9,7) node {
$\begin{array}{c c c} \occ & \vac & \vac \\ \vac & \occ & \occ\\ \hline 3 & 1 & 2 \end{array}$
};
\draw (9.2,7.9) node [color=red] {$(x_{2})$};
\draw (0,7) node 
{$
\begin{array}{c c c} \vac & \vac & \occ \\ \occ & \vac & \occ\\ \hline 2 & 3 & 1 \end{array}$
};
\draw (1.2,7.2) node [color=red] {$(x_{1})$};
\draw (-3,7) node 
{
$\begin{array}{c c c} \vac & \occ & \vac  \\ \occ & \vac & \occ\\ \hline 2 & 3 & 1 \end{array}$
};
\draw (-2.9,7.9) node [color=red] {$(x_{2})$};
\draw (6,2) node {
$\begin{array}{c c c} \occ & \vac & \vac \\ \occ & \vac & \occ\\ \hline 1 & 3 & 2 \end{array}$
};
\draw (7.2,2.2) node [color=red] {$(x_{1})$};
\draw (4.3,0) node 
{$
\begin{array}{c c c} \vac & \vac & \occ \\ \occ & \occ & \vac\\ \hline 1 & 2 & 3 \end{array}$
};
\draw (-1.2,2.2) node [color=red] {$(x_{1})$};
\draw (0,2) node 
{
$\begin{array}{c c c} \vac & \occ & \vac  \\ \occ & \occ & \vac\\ \hline 2 & 1 & 3 \end{array}$
};
\draw (0.5,0.2) node [color=red] {$(x_{1})$};
\draw (1.7,0) node {
$\begin{array}{c c c} \occ & \vac & \vac \\ \occ & \occ & \vac\\ \hline 1 & 2 & 3 \end{array}$
};
\draw (5.4,0.2) node [color=red] {$(x_{2})$};

\draw [->] (2,8.8) -- node [auto,swap] {$x_{2}$} (0.3,7.7);
\draw [->] (4,8.8) -- node [auto] {$x_{1}$} (5.7,7.7);
\draw [->] (0,6.3) -- node [auto,swap] {$x_{1}$} (0,3);
\draw [<-] (1.6,0.7) -- node [auto] {$x_{1}$} (0,1.3);
\draw [<-] (2.2,0.7) -- node [auto,very near end] {$x_{2}$} (5.1,1.7);
\draw [<-] (6,3) -- node [auto,swap] {$x_{1}$} (6,6.3);
\draw [->] (1.9,0.7) -- node [auto,very near end] {$x_{1}$} (2.9,8.3);
\draw [->] (5.7,2.9) to [out = 90, in = -90] node [auto,very near end,swap] {$x_{1}$}(0.3,6.3);
\draw [->] (0.3,2.9) to [out = 90, in = -90] node [auto,very near end] {$x_{2}$}(5.7,6.3);

\draw [->] (6.8,7) -- node [auto] {$x_{2}$} (8,7);
\draw [->] (-0.8,7) -- node [auto,swap] {$x_{2}$} (-2,7);
\draw [->] (2.5,0) -- node [auto,swap] {$x_{2}$} (3.5,0);

\draw [->] (8.5,6.3) -- node [auto] {$x_{1}$} (6.3,2.9);
\draw [->] (-2.7,6.3) -- node [auto,swap] {$x_{1}$} (-0.3,2.9);
\draw [->] (4.1,0.7) -- node [auto,very near end, swap] {$x_{1}$} (3.2,8.3);

\end{tikzpicture}
\end{center}
\caption{All transitions and the corresponding rates for the inhomogeneous Ferrari Martin multiline process with $m=(1,1,1)$. The stationary weights are given in parenthesis and in red. Compare this with Figure~\ref{F:permeg} to see the lumping procedure.}
\label{F:mlq1eg}
\end{figure}

We now  collect some observations about these transitions in the following lemma. Each of these can be verified without too much difficulty and we skip the proof.

\begin{lemma} \label{L:prop3species}\hspace{1cm}
\begin{enumerate}
\item If $\bp\con_{i}=2$, the multiline queue at site $i$ is forced to be of the form
\bew
\begin{array}{c}
\vac \\
\occ \\
\hline
2
\end{array}
\eew
irrespective of the label at sites before and after $i$.

\item If $\bp\con_{i-1}=1$ and $\bp\con_{i}=2$, the ringing path transition at site $i$ does nothing.
\bew
\begin{array}{cc}
\vac/\occ & \vac \\
\occ & \occ \\
\hline
1 & 2
\end{array}
\eew

\item Suppose $\bp\con_{i}, \bp\con_{i+1}, \dots, \bp\con_{j}=3$ and $\bp\con_{i-1},\bp\con_{j+1}\neq 3$. Then there is at most one non-covered 3 at position $l$ in this block such that the ringing path transition at $l$ changes the configuration.

\item The number of covered 3's can increase (resp. decrease) only with transitions starting at sites with non-covered 3's (resp. 1's).
\end{enumerate}
\end{lemma}

\begin{theorem} \label{T:FM3}
Let $m=(m_{1},m_{2},m_{3})$. With transition rates as above the stationary weight of a configuration
$\con\in\ofm_m$ is $x_{1}^{m_{3}-k} x_{2}^k$, where $k$ 
is the number of covered 3's in $\con$. 
Moreover, the Ferrari-Martin multiline process lumps to the inhomogeneous multispecies exclusion process.
\end{theorem}

\begin{proof}
By \refT{T:FMbully} (3) , the only possible effect on the bully-path projection of the ringing path transition on $\con$ is that neighboring particles exchange according  multispecies exchange rules. By Definition~\ref{D:FMq3}, these happen with the same rates as the multispecies
exclusion process, namely $x_{1}$ and $x_{2}$. When a nontrivial ringing path transition 
is initiated at a position with a 3 (covered or not), the bully-path projection does not change.
This will guarantee that the lumpability criterion is satisfied. 

We will prove the result by showing that the formula  for the stationary weight satisfies the master equation \eqref{mastereq} at each state. Equivalently, we want to show that the total probability leaving a given configuration $\con$ is equal to that entering it. This is easiest if we represent the corresponding multipermutation in block form,
\bew
\bp\con = 1^{a_{1}} 2^{b_{1}}3^{c_{1}} 1^{a_{2}} 2^{b_{2}}3^{c_{2}} \cdots 
1^{a_{j}} 2^{b_{j}}3^{c_{j}},
\eew
where we want this representation to be maximal and unique, so that one of $a_{i}, b_{i}$ or $c_{i}$ can be zero but not two consecutive one of these can be zero. Further, suppose that the number of non-covered 3's in $\con$ is $k$, so that the stationary weight of $\con$ is 
$x_{1}^{m_{3}-k} x_{2}^k$.

Recall that by Theorem \ref{T:FMringing}(1) we have as many incoming as outgoing transitions for each state. Below we will ignore the loops and count only transitions where the state changes.
Let us first analyze the $x_{2}$ transitions. These occur if the ringing path transition starts at a site $i$ such that $\bp\con_{i}=2$ or a non-covered 3. By \refL{L:prop3species}(1), if $\bp\con_{i}=2$, the transition can only occur if $\bp\con_{i-1}=3$. Therefore, all such transitions occur either inside a block of 3's or at its boundary. Let us first look at outgoing transitions from 
$\con$ focussing our attention on a block of 3's. Notice that if the leftmost 3 in the block is covered, then so are all the other 3's in that block. This also forces the site to the right of that block to be a 1. Therefore, no $x_{2}$ transition can take place. The only $x_{2}$ transitions that occur happen when the leftmost 3 is not covered. We now have two possibilities.
\begin{enumerate}
\item[(a)] There is a transition at a site containing a 2, which happens only when all the sites in the first row of the block are $\vac$'s.
\bew
\begin{array}{ccccc}
\vac & \vac & \dots & |\vac & \vac| \\
\vac & \vac & \dots & |\vac & \occ| \\
\hline
3 & 3 & \dots & 3 & 2 
\end{array}
\longrightarrow
\begin{array}{ccccc}
\vac & \vac & \dots & |\vac & \vac| \\
\vac & \vac & \dots & |\occ & \vac| \\
\hline
3 & 3 & \dots & 2 & 3 
\end{array}
\eew

\item[(b)] There is a transition at a site containing a 3 within the block (the underlined 3 below). By \refL{L:prop3species}(3), only one such transition can occur. 
\bew
\begin{array}{ccccccc}
\!\dots\!&\vac & |\vac\; & \;\occ|&\!\dots & \vac/\occ & \vac/\occ\\
\!\dots\!&|\vac\; & \;\vac| & \vac & \!\dots & \vac & \occ\\
\hline
\!\dots\!&3 & \underline{3} & 3 &\!\dots & 3 & 1
\end{array}
\longrightarrow
\begin{array}{cccccccc}
\!\dots\!&\vac & |\occ\; & \;\vac|&\!\dots & \vac/\occ & \vac/\occ\\
\!\dots\!&|\vac\; & \;\vac| & \vac &\!\dots & \vac & \occ\\
\hline
\!\dots\!&3 & 3 & 3 &\!\dots & 3 & 1
\end{array}
\eew
\end{enumerate}

The important fact is that the (a) and (b) transitions are mutually exclusive as well as exhaustive. Therefore, there is one $x_{2}$ 
transition for every block of 3's, the leftmost of which is not covered and no other $x_2$ transitions.

\smallskip
Now, let us look at incoming transitions in the master equation \eqref{mastereq}. It will be convenient
to use the notion of effective rate defined in equation \eqref{effrate}.
We want to show that the only transitions with an effective rate of $x_{2}$ are exactly one per block of 3's with a non-covered leftmost 3. By effective rate $x_2$, we mean contributing $x_2w(\con)$ to the incoming side of the master equation, which happens in two different ways. Either the configuration leading to $\con$ has $k$ covered 3's and the transition has rate $x_{2}$ or it has $k+1$ covered 3's, i.e. stationary weight $x_1^{m_3-k-1}x_2^{k+1}$ and the transition has rate 
$x_{1}$. 
This time, the possibilities depend on the site to the left of the block in $\con$.
This is because a nontrivial ringing path transition at a site occupied by a non-covered 3 (after the transition) comes from a state with fewer covered 3's and a 
transition at a covered 3 has rate $x_1$ and comes from a state with the same weight.
To differentiate between incoming and outgoing transitions, we always place $\con$ on the left in the figures below.
\begin{enumerate}
\item[(A)] If the jumping particle is a 2, then the transition happens with rate $x_{2}$ and the number of covered 3's is unchanged.
\bew
\begin{array}{ccccc}
|\vac\; & \;\vac| & \cdot & \dots & \cdot\\
|\occ\; & \;\vac| & \vac & \dots & \vac\\
\hline
2 & 3 & 3 & \dots & 3 
\end{array}
\longleftarrow
\begin{array}{ccccc}
|\vac\; & \;\vac| & \cdot & \dots & \cdot\\
|\vac\; & \;\occ| & \vac & \dots & \vac\\
\hline
3 & 2 & 3 & \dots & 3 
\end{array}
\eew

\item[(B)] If the jumping particle is a 1, then the transition happens with rate $x_{1}$ and the number of covered 3's decreases by 1 (since we assumed the leftmost 3 of the block of $\con$ to be non-covered).
\bew
\begin{array}{ccccc}
|\vac/\occ\; & \;\vac| & \cdot & \dots & \cdot\\
|\occ\; & \;\vac| & \vac & \dots & \vac\\
\hline
1 & 3 & 3 & \dots & 3 
\end{array}
\longleftarrow
\begin{array}{ccccc}
|\vac/\occ\; & \;\vac| & \cdot & \dots & \cdot\\
|\vac\; & \;\occ| & \vac & \dots & \vac\\
\hline
3 & 1 & 3 & \dots & 3 
\end{array}
\eew
There is a subtle point here, which we should emphasize. Note that the configuration on the right has a $\vac$ at the first row in the column corresponding to 1. If there is a $\occ$ at that position, that configuration may also go to $\con$ with rate $x_{1}$. However, the number of covered 3's will not change, and the transition will have an effective rate of $x_{1}$.  
\end{enumerate}

We have thus shown that there is exactly one incoming transition with an effective rate of $x_{2}$ for every block with a leading non-covered 3. We now show that all other incoming transitions have effective rate $x_1$. 

If a non-trivial transition starts at a covered  3 then the effective rate is $x_1$ by the discussion before (A) above. If a 1 has jumped left over a 2 or a non-covered 3 the effective rate is clearly $x_1$. 
 The last possibility to consider is when a 1 has jumped over the 3, and the 3 is covered after
the jump, as shown  below. 
\bew
\begin{array}{ccccc}
|\vac/\occ & \vac/\occ| & \cdot   & \dots & \cdot\\
|\occ & \vac |  & \vac & \dots & \vac\\
\hline
1 &  3 & 3 & \dots & 3 
\end{array}
\longleftarrow
\begin{array}{cccccc}
|\vac/\occ & \vac/\occ | &  \cdot & \dots & \cdot\\
|\vac & \occ | & \vac & \dots & \vac\\
\hline
3 & 1 & 3 & \dots & 3 
\end{array}
\eew

This happens with rate $x_{1}$ and the 3 must have been covered both before and after
the jump so the effective rate is also $x_1$.
Note that the 3 may be covered by a bully path starting either directly above it or somewhere to the left outside this picture.

To complete the proof, we use that the total number of incoming transitions equals that of the outgoing transitions; see Theorem \ref{T:FMringing}(1).
\end{proof}

\section{Ferrari-Martin process with one first class particle} \label{S:FM1}
Our next result is for arbitrary $n$ with the condition $m_{1}=1$. 
We also fix the rates so that we only have one free parameter.
Recall that $z_{r,i}(\con)$ is the number of $\vac$'s on row $r$ of $\con$  that are $i$-covered and that $v_1$ is the
total number of $\vac$'s below row 1. Let $z_{1}(\con) = \sum_{r}z_{r,1}(\con)$.

\begin{theorem}\label{T:FM1}
In $\ofm_{m}$ with $m_{1}=1$ and $x_i=1$ for $2\le i\le n$, set the rate of the Ferrari-Martin transition from $\con$ at column $i$ to take place with rate $x_{1}$ when 
$\bp\con_{i}=1$. If $\bp\con_{i}=2,\dots,n$, then set the rate to be 1.
Then the stationary probability of $\con$ is $w(\con)=\frac{x_{1}^{v_1-z_{1}(\con)}}{Z_m}$, where $Z_m$ is the normalizing partition function. 
\end{theorem}

\begin{proof}
Let the probability of configuration $\con$ be given by $w({\con})$. Since there is a unique solution of the master equation, it suffices to show that 
$w({\con})=x_{1}^{-z_{1}(\con)}$ satisfies  \eqref{mastereq}.

By \refT{T:FMringing}(3), when all the rates are 1, there are as many incoming transitions into any configuration as outgoing ones, say $k$. Therefore, we need to compare the number of incoming transitions with an effective rate (see \eqref{effrate}) of $x_{1}$ with the number of outgoing transitions with rate $x_{1}$. Since there is only one first class particle, the only time we get an outgoing transition with rate 
$x_{1}$ is if the ringing path transition occurs at the unique site $i$ so that 
$\bp\con_{i}=1$. There is thus a contribution of $(x_{1}+k-1)w({\con})$ to the outgoing transitions from $\con$.

Now, let us look at incoming transitions to $\con$ in the master equation \eqref{mastereq}.
The key observation is illustrated by Figure~\ref{F:bullypath-one1class}, which is a cartoon for a configuration $\con'$ along with a generic ringing path transition at $j$ to $\con$.
\begin{figure}[h!]
\begin{center}
\includegraphics[height=4cm]{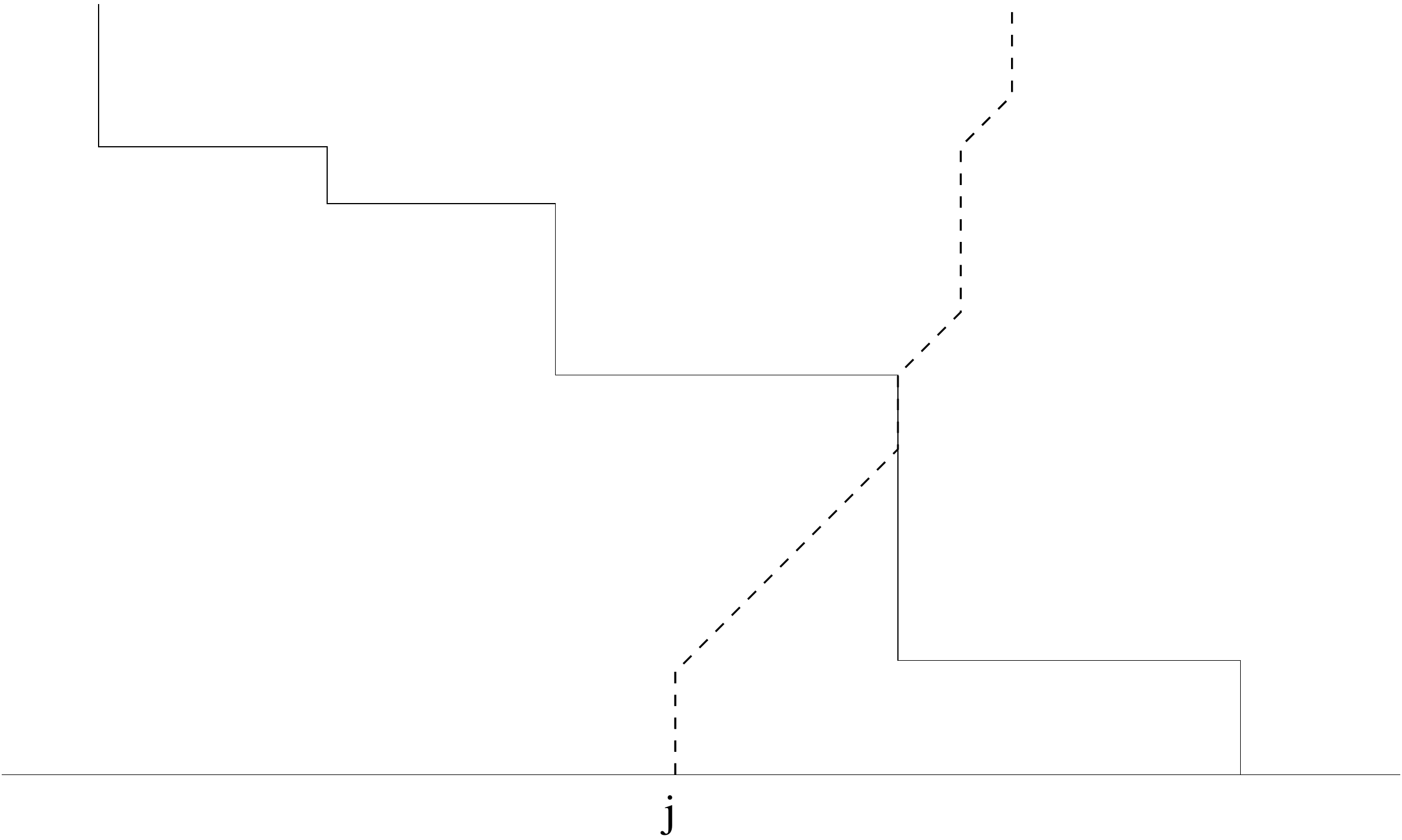}
\caption{The (dashed) ringing path starting at $j$ intersecting with the (solid) first class bully path.  This transition has rate 1.}
\label{F:bullypath-one1class}
\end{center}
\end{figure}

The only ringing transitions that affect the bully path are at the locations where both paths cross. They cross either in a single vacancy $\vac$ or in a vertical tower of $\occ$'s as shown in \refF{F:bullypath-one1class}. In the first case the transition does not change the bully path. 

When the transition takes place in the latter case, all the $\occ$'s at the intersection will move left, if possible. If any particular $\occ$ did not move, that means there is already a $\occ$ in the position to its left. Therefore, the vertical part of the bully path will shift by one to the left in $\con$. This means that one less $\vac$ is 1-covered at the top of the vertical part and one more $\vac$ is 1-covered at the bottom by the first class bully path. Note that the ringing path could intersect the bully path multiple times because these multiline queues live on a cylinder. But the argument given above holds true for every vertical segment of the intersection of the two separately.
Therefore, $z_{1}(\con') = z_{1}(\con)$, and this transition contributes $1\cdot w({\con'}) = w({\con})$ by assumption.

We have shown that a transition whose ringing path intersects the first class bully path somewhere in the middle does not change $z_{1}$ after the transition. The only possible way for the transition to affect $z_{1}$ is if the ringing path intersects the bully path either at the first row or the last row or both. As we saw above, it could intersect the bully path at other places, but they don't matter. We shall consider these three cases separately.
\begin{enumerate}
\item {\bf Last Row only:} In this case, the ringing path in $\con'$ starts at the site containing the $\occ$ of class 1, and the transition has rate $x_{1}$. By the argument above, $\con'$ has one more 1-covered $\vac$ than 
$\con$. Therefore this transition contributes $x_{1} \cdot w({\con'}) = w({\con})$ and has an effective rate of 1.

\item {\bf First Row only:} The transition has rate 1 since the ringing path in $\con'$ starts at a site not containing a 1. But this time, there is one less 1-covered $\vac$ in $\con'$ than $\con$ and therefore the contribution is $1 \cdot w({\con'}) = x_{1} w({\con})$, i.e. with effective rate $x_1$.

\item {\bf Both First and Last Rows:} The transition has rate $x_{1}$ since the ringing path in $\con'$ starts at the site containing 1. But this time, there is one less 1-covered $\vac$ at the top and one more at the bottom in $\con'$ compared to $\con$ and therefore the contribution is $x_{1} \cdot w({\con'}) = x_{1} w({\con})$.
\end{enumerate}

This shows that the only incoming transitions which have an effective rate of $x_{1}$ are those which intersect the first class bully path at the top. By the inverse of the ringing path construction mentioned in \refT{T:FMringing}(1), see \cite[Prop 3.3]{FM2}
there is exactly one such transition. 
Thus, there are $k-1$ incoming transitions with effective rate 1 and exactly one with rate $x_{1}$.
\end{proof}

In this case we may use the theorem to compute the partition function $Z_m$, i.e. the divisor for all stationary distributions to get a probability measure with the normalization we are using. 
Recall that we have $M_{r} = \sum_{i=1}^{r} m_{i}$.
 If we let $[k]_q:=\frac{1-q^k}{1-q}$, and denote by $[k]_q^{(d)}$ the $d$'th derivative of 
 $[k]_{q}$ with respect to $q$. Then 
 $[k]_q^{(d)}=d!\sum_{i=0}^{k-d-1}\binom{i+d}{i}q^i$.

 \begin{proposition}
In $\ofm_{m}$ with $m_{1}=1$, $x_{1}=a$ and $x_i=1$ for $2\le i\le n$, the partition function is
 \[ Z_{m}=N\prod_{r=2}^{n-1} \frac{ [N]_{a}^{(M_{r}-1)}}{(M_{r}-1)!}
 = N\prod_{r=2}^{n-1} h_{n-r}(1, \underbrace{a, \dots, a}_{r}),
 \]
 where $h_{k}$ is the complete homogeneous symmetric polynomial. 
 \end{proposition}

\begin{proof}
The latter equality is an immediate consequence of the definition of the complete homogeneous symmetric polynomials. 
We will prove the former inequality combinatorially by looking at the \mlqs. 

Since we have assumed that $m_1=1$ and $x_i=1$ for $i\ge 2$ we need to focus on the bully path of the particle 1. 
For $2\le r\le n-1$, let $e_r = z_{r,1}$. The $e_r$'s determine the bully path except for the starting position. 
For every fixed set $\{e_{2},\dots,e_{n-1}\}$ where
$0\le e_r\le N-M_r$, such a bully path can be completed to a multiline queue by choosing, for each row $r$,  the remaining $M_r-1$ positions of particles among the $N-e_r-1$ possible positions. 
Since the 1 takes  highest precendence in every queue this can be done 
independently for every row. 
Let $k:=\sum_{r=1}^n e_r$ and 
$Z_{m,k}:=\sum_{\ofm_m} a^k$. 
For every vacancy $\vac$ that the bully path goes through, 
the stationary weight gets a factor $a^{-1}$ by \refT{T:FM1}. The starting position of the bully path gives a factor $N$ and together we get 
\[
Z_{m,k}=N\cdot a^{v_1}\sum_{\substack{e_2,\dots,e_{n-1}\\ 0\le e_r\le N-M_{r}\\ k=\sum e_r}} 
\binom{N-e_2-1}{M_2-1}a^{-e_2}
\dots
\binom{N-e_{n-1}-1}{M_{n-1}-1}a^{-e_{n-1}}
\]
To get the partition function, we use $Z_{m}=\sum_k Z_{m,k}$. We get by standard manipulations the following formula. The easiest way to see the identity is to expand the product below, which gives all the possible terms above. 

 \[ Z_{m}=
 Na^{v_1} \prod_{r=2}^{n-1} \left(\binom{N-1}{M_{r}-1}+\binom{N-2}{M_{r}-1}a^{-1}+\dots+
 \binom{M_{r}-1}{M_{r}-1}a^{-(N-M_{r})}\right)
 \]

Using that $v_1=\sum_{i=2}^{n-1} (N-M_i)$,  we can rewrite the formula above as 

\[
  Z_{m}=Na^{v_1}\prod_{r=2}^{n-1} \sum_{i=0}^{N-M_{r}}\binom{N-i-1}{M_{r}-1}a^{-i}
  =N\prod_{r=2}^{n-1} \sum_{i=0}^{N-M_{r}}\binom{M_r+i-1}{M_{r}-1}a^{i},
\]
which can be seen to be the same as the desired formula. 
\end{proof}
 
 \section{A minimal Markov chain for $n=3$} \label{S:new}
 As mentioned after Conjecture \ref{C:main}, for $n\ge 4$ we do not think that there is a simple rule for transition 
 rates in general for the ringing path transitions in $\ofm_m$
 that would give each \mlq{} a monomial as stationary weight.  
 For example, when $N=n=4$, assuming all transition rates are of the form $x_{i}$ and all stationary weights are monomials, all master equations cannot be satisfied simultaneously.
 
 We have therefore sought other transitions between \mlqs{} to create other
 Markov chains for which we might be able to find appropriate rates. We will in this section present such a Markov chain $\oal$ for the case
  $n=3$. It gives the same stationary weights as \refT{T:FM3} and Conjecture \ref{C:main} and is thus an independent solution of the
  TASEP for $n=3$, parallel to the work by Angel \cite{A}. For $\oal$ we believe that it might be possible to generalize to higher $n$, but we have not yet been able to do so in general. We will continue to use the bully-path projection defined in 
Algorithm~\ref{A:bullypath} as the lumping procedure.
 
We call $\oal$ the {\bf Multiline coupe process}. 
It will be minimal in the sense that there will be no transitions between different configurations corresponding to the same permutation.

First we will need some definitions. In this section $m=(m_1,m_2,m_3)$. We will divide a word $w\in\om_m$ into pieces. We cut 
$w$ to the right of every 2 not followed by a 2 and to the right of every 1 that is not followed by a 1. Every subword thus obtained will start with a number (possibly zero)
of 3's followed by some consecutive 1's or consecutive 2's, see \refF{F:Coupe}. 
Each such subword will be called a {\bf coupe}.  
Each coupe contains exactly one among 1's or 2's and are naturally called {\bf first-class} and {\bf second-class} coupes.
We call a first-class or second-class coupe with no 3 {\bf full}.

\begin{figure}[htbp]
\begin{center}
\begin{tabular}{c |c c c | c c c c | c c c c| c | c | c c c c c | c | c}
 $\cdots$ &  3&3&2&3&3&1&1&3&2&2&2&1&2&3&3&1&1&1&2& $\cdots$
\end{tabular}
\end{center}
\caption{A decomposition of a circular word into coupes, seperated by vertical lines.}
\label{F:Coupe}
\end{figure}

The partition of a word into coupes induces a partition 
of each configuration corresponding to that 
word. 
In a coupe we will call (the position of)  the leftmost 1 (or 2) the {\bf front seat} and the rightmost letter in each coupe is called 
the {\bf back seat}.
It is easy to see that the only particle in each coupe that can jump is the front seat particle. Similarly, the only way to get to a given 
configuration is if the back seat just jumped. 

If in a configuration $\con$, a back or front seat has a $\occ$ in row 1, then it is said to be {\bf occupied}, if it has a $\vac$ in row 1, then it is said to be {\bf vacant}.
The front seat particle can always jump except when the coupe consists of only 2's. In that case, the back seat of the coupe to its left is a 1, and so the 2 cannot jump. Note that a 2 in a front or back seat must always be vacant.

\begin{definition}\label{D:jumprules}
The transitions in $\oal$ are given by the following rules.
\begin{itemize}
\item{} {\bf Regular jump} - If a 1 corresponding to an occupied front seat jumps then both the $\occ$'s  in row 1 and row 2 jumps to the left if possible. 
\item{} {\bf Pulling jump} - If a 1 or a 2 corresponding to a vacant front seat jumps then all $\occ$'s  in the first row and to the right of the jumping $\occ$ but in the same coupe also moves one step to the left. If the jumping
1 or 2 corresponding to a vacant front seat is also a back seat then all the $\occ$'s in the coupe to the right are moved one step to the left.
\end{itemize}
\end{definition}

Here are two examples of regular jumps:

\bew
\begin{split}
\begin{array}{ |c c  c c |c }
 \vac&\vac &\occ & \occ&\vac\\
  \vac& \vac& \occ & \occ&\vac\\
  \hline
  3&3&1&1&3
\end{array}
&\Rightarrow
\begin{array}{ |c c |c c |c}
 \vac&\occ &\vac &\occ &\vac\\
\vac  & \occ &  \vac& \occ&\vac\\
  \hline
  3&1&3&1&3 ,
\end{array}
\\
\begin{array}{ |c c c c|c }
 \vac&\occ &\occ & \vac&\vac\\
  \vac& \vac& \occ & \occ&\vac\\
  \hline
  3&3&1&1&3
\end{array}
&\Rightarrow
\begin{array}{ |c c |c c | c }
 \vac&\occ &\occ &\vac&\vac \\
\vac  & \occ &  \vac& \occ&\vac\\
  \hline
  3&1&3&1&3 ,
\end{array}
\end{split}
\eew

followed by two examples of pulling jumps

\bew
\begin{split}
\begin{array}{ |c c c c |c }
 \vac&\occ &\vac & \occ&\vac\\
  \vac& \vac& \occ & \occ&\vac\\
  \hline
  3&3&1&1&3
\end{array}
&\Rightarrow
\begin{array}{ |c c| c c |c}
 \vac&\occ &\occ &\vac &\vac\\
\vac  & \occ &  \vac& \occ&\vac\\
  \hline
  3&1&3&1&3 ,
\end{array}
\\
\begin{array}{ |c c c | c c c c|c }
 \vac&\vac &\vac & \occ&\vac&\vac&\occ&\vac\\
  \vac& \vac& \occ & \vac&\vac&\occ&\occ&\vac\\
  \hline
  3&3&2&3&3&1&1&3
\end{array}
&\Rightarrow
\begin{array}{ |c c |c  c c c c|c }
 \vac&\vac  & \occ&\vac&\vac&\occ&\vac&\vac\\
  \vac& \occ& \vac & \vac&\vac&\occ&\occ&\vac\\
  \hline
  3&2&3&3&3&1&1&3.
\end{array}
\end{split}
\eew

Note that a jump might increase the number of coupes by one, decrease it by one, or leave it unchanged.
Note also that after the pulling jump the coupe to the right will have a vacant back seat. 

The rate when a 1 (resp. 2) jumps is $x_1$ (resp. $x_2$)  since 
we want the Multiline coupe process to be lumpable with respect to the bully-path projection. As an example of the multiline coupe process we give $\oal_m$  when $m=(1,1,1)$ in \refF{F:mlq2eg}.

\begin{figure}[h]
\begin{center}
\begin{tikzpicture} [scale=0.88, >=triangle 45]
\draw (3,9) node 
{$
\begin{array}{c c c} \vac & \vac & \occ \\ \vac & \occ & \occ\\ \hline 3 & 2 & 1 \end{array}$
};
\draw (4.2,9.2) node [color=red] {$(x_{1})$};
\draw (6,7) node 
{
$\begin{array}{c c c} \vac & \occ & \vac  \\ \vac & \occ & \occ\\ \hline 3 & 1 & 2 \end{array}$
};
\draw (4.8,7.2) node [color=red] {$(x_{1})$};
\draw (9,7) node {
$\begin{array}{c c c} \occ & \vac & \vac \\ \vac & \occ & \occ\\ \hline 3 & 1 & 2 \end{array}$
};
\draw (9.2,7.9) node [color=red] {$(x_{2})$};
\draw (0,7) node 
{$
\begin{array}{c c c} \vac & \vac & \occ \\ \occ & \vac & \occ\\ \hline 2 & 3 & 1 \end{array}$
};
\draw (1.2,7.2) node [color=red] {$(x_{1})$};
\draw (-3,7) node 
{
$\begin{array}{c c c} \vac & \occ & \vac  \\ \occ & \vac & \occ\\ \hline 2 & 3 & 1 \end{array}$
};
\draw (-2.9,7.9) node [color=red] {$(x_{2})$};
\draw (6,2) node {
$\begin{array}{c c c} \occ & \vac & \vac \\ \occ & \vac & \occ\\ \hline 1 & 3 & 2 \end{array}$
};
\draw (7.2,2.2) node [color=red] {$(x_{1})$};
\draw (4.3,0) node 
{$
\begin{array}{c c c} \vac & \vac & \occ \\ \occ & \occ & \vac\\ \hline 1 & 2 & 3 \end{array}$
};
\draw (-1.2,2.2) node [color=red] {$(x_{1})$};
\draw (0,2) node 
{
$\begin{array}{c c c} \vac & \occ & \vac  \\ \occ & \occ & \vac\\ \hline 2 & 1 & 3 \end{array}$
};
\draw (0.5,0.2) node [color=red] {$(x_{1})$};
\draw (1.7,0) node {
$\begin{array}{c c c} \occ & \vac & \vac \\ \occ & \occ & \vac\\ \hline 1 & 2 & 3 \end{array}$
};
\draw (5.4,0.2) node [color=red] {$(x_{2})$};

\draw [->] (2,8.8) -- node [auto,swap] {$x_{2}$} (-2.6,7.7);
\draw [->] (4,8.8) -- node [auto] {$x_{1}$} (5.7,7.7);
\draw [->] (0,6.3) -- node [auto,swap] {$x_{1}$} (0,3);
\draw [<-] (1.6,0.7) -- node [auto] {$x_{1}$} (0,1.3);
\draw [<-] (4.2,0.7) -- node [auto,swap] {$x_{2}$} (6,1.3);
\draw [<-] (6,3) -- node [auto,swap] {$x_{1}$} (6,6.3);
\draw [->] (1.9,0.7) -- node [auto,very near end] {$x_{1}$} (2.9,8.3);
\draw [->] (5.7,2.9) -- node [auto,very near end,swap] {$x_{1}$}(0.3,6.3);
\draw [->] (0.3,2.9) -- node [auto,very near end] {$x_{2}$}(8.4,6.3);

\draw [->] (8.5,6.3) -- node [auto] {$x_{1}$} (6.3,2.9);
\draw [->] (-2.7,6.3) -- node [auto,swap] {$x_{1}$} (-0.3,2.9);
\draw [->] (4.1,0.7) -- node [auto,very near end, swap] {$x_{1}$} (3.2,8.3);

\end{tikzpicture}
\end{center}
\caption{All transitions and the corresponding rates for the inhomogeneous Multiline coupe process with m = (1, 1, 1). The stationary weights are given in parenthesis and in red.
Compare this with Figure~\ref{F:permeg} to see the lumping procedure.}
\label{F:mlq2eg}
\end{figure}
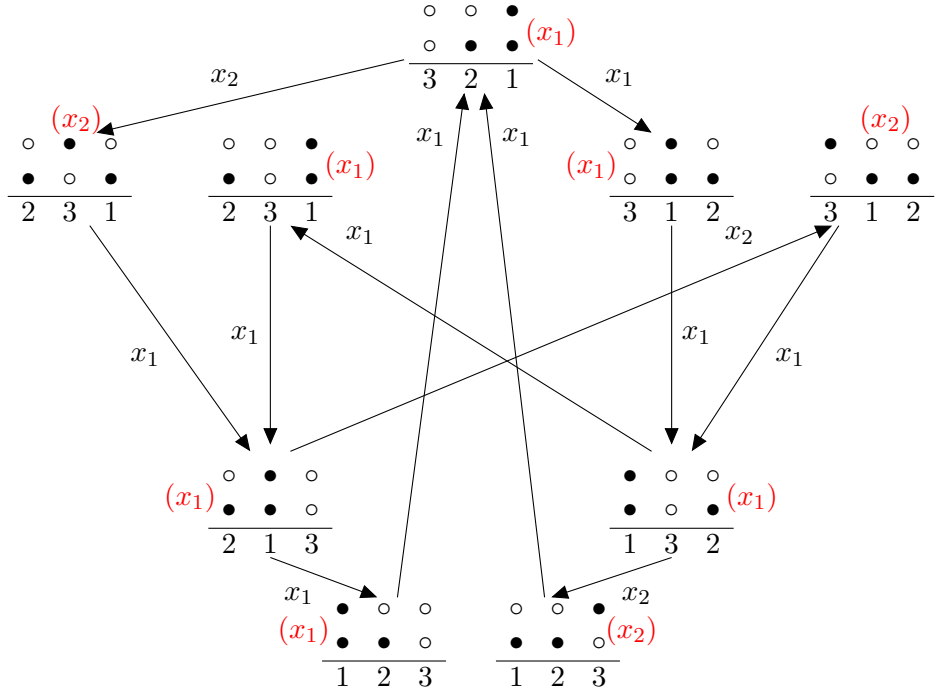

\begin{remark} The jump rules above could be thought of a sequence of ringing path transitions, but not in a straightforward way. It is possible that proofs or generalizations will be easier if viewed that way. Note also that row 1 will not behave as a TASEP on its own. It would be interesting to understand this process.
\end{remark}

The following lemma describes  the reverse, i.e. the incoming transitions to a given state. The verification of the Lemma is not difficult and we omit the details.

\begin{lemma}\label{L:inverse}
Let $\con$ be a configuration. Transitions leading to $\con$ in $\oal_m$ are determined by the back seats in the following way. 
\begin{romenumerate}
\item{} An occupied back seat can always have been the last 1-class particle to jump by a regular jump.
\item{} Any  back seat where the coupe to the right is not a full second-class coupe and has a vacant back seat could have made a transition via a pulling jump.
\end{romenumerate}
\end{lemma}

First we show that the coupe Markov chain is irreducible, i.e. that we can reach any state from any other state.
\begin{lemma}
For any $m=(m_1,m_2,m_3)$ the Multiline coupe Markov chain $\oal_m$ is irreducible.
\end{lemma}

\begin{proof} Let $w_0$ be the word starting with all the 3's, then all 2's and then all 1's and let $\con_{0}$ denote the unique \mlq{}
 corresponding to that word. First we prove that we can get from any 
state $\con$ to every cyclic shift of $\con_0$.  
Begin by moving all the 1's left until they are all directly 
to the right of a 2. Then every 1 will have $\occ$'s in both rows above it. Next, jump with the 2 and 
then all 1's to get to the next 2 and so on until all 2's are together. Then again jump with all the 2's and then all 1's performing
a rotation until the desired cyclic shift of $\con_0$ is obtained.

Conversely, we will show that we can get from some cyclic shift of $\con_0$ to any state $\con$. Equivalently, we show that we  can perform inverse jumps from any state 
$\con$ as described by \refL{L:inverse}, and reach some cyclic shift of $\con_0$.

We begin by locating a first-class coupe $X$  to the left of a second-class coupe $Y$. 
If $X$  has a vacant back seat we do inverse jumps (ii) with the back seat in the coupe to the left of $X$ 
until the back seat of $X$ is occupied. 
Then do regular inverse jumps (i) with the back seat of $X$ to the right of all the 2's in $Y$. 
Now, if there are a still a first class coupe to the left of $Y$ we repeat this procedure, moving a new 1 to the right of $Y$. 
If there are a second class coupe to the left of $Y$ we merge the two second class coupes into one by 
 inverse jumps of type (ii). Continuing this process will give us some cyclic shift of 
 $\con_0$.
\end{proof}

\begin{theorem} For any $m=(m_1,m_2,m_3)$, the jumping rules defined above form a Markov chain $\oal_m$ in which configurations have the stationary weights of \refT{T:FM3}, i.e $w(\con)=x_1^{m_3-k}x_2^k$, where $k$ is the number of covered 3's in $\con$. Moreover, this Markov chain lumps to the inhomogeneous TASEP on $\om_m$. 
\end{theorem}
\begin{proof}
Let $\con$ be a configuration of $\oal_m$ with stationary weight $w(\con)$ and which is divided into $c$ coupes. We have to verify that the master equation \eqref{mastereq} is satisfied for $\con$. We will now show that the number of incoming transitions equals the number of outgoing ones. For this purpose, it will be convenient to distinguish the coupes. Let $c_{1}$ (resp. $c_{2}$) be the number of first-class (resp. second-class) coupes. Similarly, let $f_{1}$ (resp. $f_{2}$) be the number of full first-class (resp. full second-class) coupes, and $e_{i} = c_{i}-f_{i}$. Further, let $o_{1}$ be the number of first-class coupes whose back seat is occupied, and $v_{1}$ (resp. $v_{2}$) be the number of first-class (resp. second-class) coupes whose back seat is vacant. Note that $c_{2}=v_{2}$.

In each first-class coupe exactly one particle can jump, namely the front seat. In the second-class coupes the front seat can jump if and only if it is not full. 
Thus, the total number of transitions going out of $\con$ is equal to 
$c_{1}+e_{2}$.
By \refL{L:inverse} (i) there is exactly one transition coming in  for each occupied back seat by a regular jump. There are $o_{1}$ of these. 
By \refL{L:inverse} (ii), there is also one transition coming in to 
$\con$ for each coupe unless the coupe to the right has an occupied back seat or is a full second-class coupe. This is given by $v_{1}+v_{2}-f_{2}$, and so the number of transitions coming in to $\con$ is the same as the number of transitions leaving.

We must now also check the weights. For each second-class coupe there is one transition leaving with rate $x_2$ if and only if 
the second-class coupe is not full. Each first-class coupe gives an outgoing transition with rate $x_1$.  
So the outgoing side of the master equation is 
$c_{1}x_1 \cdot w(\con)+e_{2}x_2 \cdot w(\con)$. 

Now to the incoming side of the master equation \eqref{mastereq}. We will again
use the notion of effective rate (see \eqref{effrate}). 
Any incoming transition corresponding to a regular jump has weight $x_1$ and 
the number of 1-covered $\vac$'s on row 2 has not changed so it will come from a configuration with the same stationary weight. 
The total incoming contribution from regular jumps is thus $o_{1}x_1 \cdot w(\con)$. 
For the incoming pulling jumps there are several cases to consider: 

I. A back seat in a first class coupe with a first-class coupe to the right with a vacant back seat. By \refL{L:inverse} (ii) this corresponds 
to an incoming transition with weight $x_1$. Because of the definition of a pulling jump, the number of 1-covered $\vac$'s in row 2 do not 
change (either it is covered both before and after, or neither before nor after), so the transition comes from a configuration with the same 
stationary distribution and thus contributes $x_1w(\con)$. The effective rate is thus $x_1$.

II. A back seat in a second class coupe with a first-class coupe to the right with a vacant back seat. This corresponds to an incoming transition with weight $x_2$. Before the jump there was one less 1-covered $\vac$'s in row 2 in the first class coupe and thus the weight of the previous configuration was $w(\con)\cdot \frac{x_1}{x_2}$. The contribution from this transition is 
$x_2\cdot w(\con) \frac{x_1}{x_2}=x_1w(\con)$. 
The effective rate is thus $x_1$.

III. A back seat in a first-class coupe at position $i$ with a second-class coupe to the right that is not full. By 
\refL{L:inverse} (ii) this corresponds to an incoming edge with weight $x_1$. This time the number of 1-covered 
$\vac$'s in row 2 does change. Namely the $\vac$ in row 2 at position $i+1$ is not 1-covered in $\con$,
but was 1-covered in the previous configuration, which thus had stationary weight $w(\con)\cdot \frac{x_2}{x_1}$.
The contribution for this edge is thus $x_1\cdot w(\con) \frac{x_2}{x_1}=x_2w(\con)$. 

IV. A back seat in a second-class coupe that has another (not full) second-class coupe to the right. This 
incoming transition 
will have weight $x_2$ and come from a configuration with stationary distribution $w(\con)$ since the
number of 1-covered $\vac$'s has not changed. Hence contribution $x_2w(\con)$.

The number of transitions with effective rate $x_1$ (cases I and II) is $v_1$ and with effective rate $x_2$ (cases III and IV) 
is $c_{2}-f_{2}$.
By \refL{L:inverse} these are all possibilities, and it is easy to check that the master equation is satisfied for $\con$.

Finally, we must prove that $\bp$ lumps $\oal_m$ to the inhomogeneous multispecies exclusion process on $\om_m$. 
To see this it suffices to note that for every $\con\in\oal_m$ with
$\bp(\con)=w$ where $w$ has a transition to some other word $w'$, there is a unique transition from $\con$ to $\con'$ in $\oal_m$ such that $\bp(\con')=w'$ and this is clearly true by the definition of the jump rules. 
\end{proof}


\begin{thebibliography}{ABX}

\bibitem{Aas} Erik Aas, Stationary probability of the identity for the TASEP on a ring,  \texttt{arXiv:1212.6366}. 


\bibitem{A} 
 Omer Angel,  The stationary measure of a 2-type totally asymmetric exclusion
process, 
{\it J.~Comb.~Theory A} {\bf 113}, 625 (2006).

\bibitem{AM} Chikashi Arita and Kirone Mallick, Matrix product solution to an inhomogeneous multi-species TASEP, 
{\em Journal of Physics A: Mathematical and Theoretical}, {\bf 46}, No. 8 (2013).

\bibitem{BE}    
Richard A. Blythe and Martin R. Evans,
Nonequilibrium steady states of matrix product form: 
A solver's guide,
{\it J. Phys. A} \textbf{40}, R333 (2007).


\bibitem{CW1} Sylvie Corteel and Lauren Williams, 
     Tableaux combinatorics for the asymmetric exclusion process, {\em
       Advances in Applied Mathematics}, {\bf 39} no. 3 (2007), 293--310

\bibitem{CW2} Sylvie Corteel and Lauren Williams, 
Tableaux combinatorics for the asymmetric exclusion process and Askey-Wilson polynomials,
{\em Duke Math J.}, {\bf 159} no. 3 (2011), 385--415.



\bibitem{DJLS}    
  Bernard Derrida, Steven A. Janowsky, Joel L. Lebowitz and  Eugene R. Speer,
  {  Exact solution of the totally  asymmetric exclusion
 process: shock profiles}, {\it  J. Stat. Phys.}  {\bf 73}, 813  (1993).

\bibitem{DS} Enrica Duchi and Gilles Schaeffer, 
A combinatorial approach to jumping particles: The parallel TASEP,
{\em Random Structures and Algorithms}, {\bf 33} no. 4
(2008),434--451.

\bibitem{EFM}  
 Martin R. Evans, Pablo A. Ferrari and Kirone Mallick, 
 { Matrix Representation of the Stationary Measure
  for the Multispecies TASEP},
{\it  J. Stat. Phys.} \textbf{135}, 217 (2009). 


\bibitem{FFK}
Pablo A.~Ferrari, Luiz R.~G.~Fontes and Yoshiharu Kohayakawa, 
Invariant measures for a two-species asymmetric process, 
 {\it J.~Stat.~Phys.}  {\bf 76}, 1153 (1994).

\bibitem{FM1} Pablo A.~Ferrari and James B.~Martin,
 Multiclass processes,
 dual points and M/M/1 queues, {\it  Markov Proc.  Rel.  Fields}  {\bf 12},
  175  (2006).

\bibitem{FM2}  
Pablo A. Ferrari and James B. Martin,
Stationary distributions of multi-type totally asymmetric exclusion processes,
{\it Ann. Prob.}  \textbf{35}, 807 (2007). 


\bibitem{lam} Thomas Lam, 
The shape of a random affine Weyl group element, and random core partitions, {\it preprint}
\texttt{arXiv:1102.4405}

\bibitem{lamwill} Thomas Lam and Lauren Williams, 
A Markov chain on the symmetric group which is Schubert positive?, {\it Experimental Mathematics}, {\bf 21}, no 2 
(2012), 189--192. 

\bibitem{LPK} H.-W.~Lee, V.~Popkov and D.~Kim,
Two-way traffic flow: Exactly solvable model of traffic jam,
{\it J. Phys. A} \textbf{30}, no. 24 (1997), 8497--8513.


\bibitem{levperwil} David A. Levin, Yuval Peres,  and Elizabeth L. Wilmer, 
Markov chains and mixing times,
{\em American Mathematical Society}, Providence, RI (2009).

 \bibitem{Lig2} 
 Thomas M.~Liggett,  
{\em  Stochastic Models of Interacting  Systems:Contact, Voter
 and Exclusion Processes}, Springer-Verlag  New-York,  (1999).

\bibitem{LM} Svante Linusson and James Martin, {\em Stationary probabilities for an inhomogeneous multi-type TASEP}, in preparation.

 \bibitem{MGP} 
 Carolyn T.~MacDonald, Julian H.~Gibbs,  Allen C.~Pipkin, 1968,
   { Kinetics of biopolymerization on nucleic acid templates},
 {\it  Biopolymers}  {\bf 6}, 1 (1968).

\bibitem{PEM}  
Sylvain Prolhac, Martin R. Evans, Kirone Mallick,
The matrix product solution of the multispecies
partially asymmetric exclusion process,
{\it J. Phys. A} \textbf{42}, 165004 (2009).

\bibitem{Spitzer}   Frank Spitzer, 
 {Interaction of Markov Processes}, {\it  Adv. in Math.}  {\bf 5},  246 (1970).

\end{thebibliography}
\end{document}